\documentclass[english,reqno]{amsart}

\usepackage{amsmath, appendix, ulem, amsthm}
\usepackage[nobysame]{amsrefs}
\usepackage{amssymb, color}
\usepackage[margin=1in]{geometry}
\usepackage{mathrsfs}
\usepackage{graphicx}
\usepackage{subfig}
\usepackage{float}
\usepackage{epsf}
\usepackage[colorlinks=true]{hyperref}
\usepackage{titletoc}
\usepackage{epstopdf}

\graphicspath{{../Figures/}}

\numberwithin{equation}{section}

\newtheorem{lem}{Lemma}[section]
\newtheorem*{thm*}{Main Theorem}
\theoremstyle{remark}
\newtheorem{rmk}{Remark}[section]

\renewcommand{\tilde}{\widetilde}
\renewcommand{\hat}{\widehat}

\newcommand{\f}{\frac}
\newcommand{\p}{\partial}
\newcommand{\beq}{\begin{equation}}
\newcommand{\eeq}{\end{equation}}	
\newcommand{\nn}{\nonumber}
\newcommand{\R}{{\mathbb R}}

\newcommand{\del}{\partial}
\newcommand{\Denote}{\stackrel{\Delta}{=}}

\newcommand{\da}{ \, {\rm d} a}
\newcommand{\dm}{ \, {\rm d} m}
\newcommand{\dt}{ \, {\rm d} t}
\newcommand{\dx}{ \, {\rm d} x}

\newcommand{\dz}{ \, {\rm d} z}
\newcommand{\dxi}{ \, {\rm d} \xi}
\newcommand{\dv}{ \, {\rm d} v}

\newcommand{\dBt}{ \, {\rm d} B_t}

\newcommand{\Eps}{\epsilon}
\newcommand{\Ni}{\noindent}

\newcommand{\BigO}{{\mathcal{O}}}

\newcommand{\CalD}{{\mathcal{D}}}

\newcommand{\CalL}{{\mathcal{L}}}

\newcommand{\VecL}{{\CalL}}

\newcommand{\Span} {{\rm Span} \, }
\newcommand{\NullL} {{\rm Null} \, \VecL}

\newcommand{\RR}{\mathbb{R}}

\newcommand{\VV}{\mathbb{V}}

\newcommand{\abs}[1]{\left\lvert#1\right\rvert}

\newcommand{\vint}[1]{\left\langle#1\right\rangle}

\newcommand{\vpran}[1]{\left(#1\right)}

\begin{document}

\title{Mathematical Modeling and Analysis of Fractional Diffusion Induced by Intracellular Noise} 
%
\author{Weiran Sun, Min Tang and Xiaoru Xue
}
\address{Department of Mathematics, Simon Fraser University, 8888 University Dr., Burnaby, BC V5A 1S6, Canada.
\newline
This author is partially supported by NSERC Discovery Grant No. R611626.}

\address{Institute of natural sciences and department of mathematics, 
Shanghai Jiao Tong University, Shanghai, 200240, China. 
\newline
This author is partially supported by NSFC 11871340
and 91330203.}
\date{\today}
\maketitle

\begin{abstract}

In this paper we use an individual-based model and its associated kinetic equation to study the generation of  long jumps in the motion of E. coli. These models relate the run-and-tumble process to the intracellular reaction where the intrinsic noise plays a central role.
Compared with the previous work in \cite{PST} in which the parametric assumptions are mainly for mathematical convenience and not well-suited for either numerical simulation or comparison with experimental results, our current paper make use of biologically meaningful pathways and tumbling kernels. Moreover, using the individual-based model we can now perform numerical simulations. 
Power-law decay of the run length, which corresponds to L\'{e}vy-type motions, are observed in our numerical results. The particular decay rate agrees quantitatively with the analytical result. We also rigorously recover the fractional diffusion equation as the limit of the kinetic model. 

 
\end{abstract}
\bigskip

\noindent{\makebox[1in]\hrulefill}\newline

\section*{Introduction}
E.coli is known to move by alternating forward-moving "runs" and reorienting "tumbles"\cite{BB}. The switching between "runs" and "tumbles" is controlled by the rotational direction of the flagella on the cell surface of E.coli. 
Recently, biologists have uncovered the mechanism used by the biochemical pathways to regulate the flagellar motors. Since then models relating the intra-cellular molecular content with the tumbling frequency have been established \cite{JOT,STY,SWOT,X}, which we briefly explain below.

The response of the bacteria to external signal changes consists of two steps: first via ``excitation", which is a rapid change in the tumbling frequency when sensing an attractant or repellent in the environment, and then by a slow "adaption" which allows the cell to subtract the background signal and turn the tumbling frequency back to some base value. Both excitation and adaptation processes are controlled by the so-called receptor activity.  Denoted by $a$, the receptor activity depends on the intracellular methylation level $m$ and the extracellular ligand concentration $[L]$ in the way that
\beq\label{eq:adef}
a=\Big(1+\exp\Big(N(-\alpha(m-m_0)+f_0([L]))\Big)\Big)^{-1},
\qquad
f_0([L])=\ln\Big(\frac{1+[L]/K_I}{1+[L]/K_A}\Big).
\eeq
Here $N$, $\alpha$, $m_0$, $K_I$, $K_A$ are measurable constants with their biological meanings explained in \cite{SWOT}. 
The tumbling frequency $\Lambda$ for E.coli depends $a(m,[L])$ through the relation
\beq\label{eq:Lambdaa}
\Lambda(a)=\lambda_0+\tau^{-1}(a/a_0)^H,
\eeq  
where the parameters $\lambda_0$, $H$, $\tau$, $a_0$ represent respectively the rotational diffusion, the Hill coefficient of a flagellar motor's response curve, the average run time and the receptor's preferred activity. Combination of ~\eqref{eq:adef} and ~\eqref{eq:Lambdaa} gives a way to quantify the dependence of the tumbling frequency of E. Coli on the intracellular content and the exterior chemical concentration.
Intracellular adaptation dynamics of E. Coli can also be described by 
\beq\label{eq:fadef}
\frac{\dm}{\dt}=f(a)=k_R(1-a/a_0) \,,
\eeq
where $k_R$ is the adaptation time. 
In general, the specific forms of $f(\cdot)$ and $\Lambda(\cdot)$ may change depending on the types of bacteria~\cite{JOT,OXX} and 
 the frequency $\Lambda$ usually has a steep transition in $a$.


In this paper, we are interested in the population dynamics of E.coli. The particular model we use is the following bacterial run-and-tumble kinetic equation with biochemical pathway proposed in \cite{SWOT}:
\beq\label{eq:model-no-noise}
\del_t p
+ v \cdot \nabla_x p
+ \del_m \vpran{f(a)p}
= \Lambda(a) (\vint{p} - p) \,.
\eeq
Here $p(t,x,v,m)$ denotes the density function of the bacteria at time $t$, position $x\in\mathbb{R}^d$, methylation level $m \in \R$ and velocity $v\in\VV$, where $\VV$ denotes the sphere $\p B(0, v_0) \subseteq \R^d$. The velocity average $\vint{p}$ is defined by
\begin{align*}
    \vint{p}(t,x,m) = \int_\VV p (t,x,v,m) \dv \, ,
\end{align*}
where $\dv$ is the normalized surface measure such that $\int_\VV 1 \dv = 1$. The right-hand side of \eqref{eq:model-no-noise} describes the velocity jump process.

Many macroscopic models have been recovered from ~\eqref{eq:model-no-noise}. For example, in the regime where the  gradient of the ligand concentration $[L]$ is small, the classical Keller-Segel equations are derived in \cite{ErbanOthmer04, ErbanOthmer05, STY, X}. When the chemical gradient $[L]$ is large, by comparing the stiffness of response and the adaptation time, flux-limited Keller-Segel models are derived \cite{PSTY,ST}, which give an explanation of the phenomena that the drift velocity on the population level should be bounded. The diffusion terms in these models suggest that in these parameter regimes, 
the underlying microscopic dynamics of the bacteria follow a Brownian motion. 

The above theoretical results can be compared with experiments, since nowadays biologists are able to track the trajectories of each individual cell. By recording the run lengths between two successive tumbles, one can find the path length distribution of the run duration. For a Brownian motion such distribution should have a fast decay at long distances. In \cite{Harris, ARBPHB}, however, it is found that the path length distributions of some bacteria or cells actually obey a slow power-law decay. This suggests that instead of the Brownian motion, some bacteria adopt L{\'e}vy-flight type movement and have a non-negligible probability of making long jumps. 
In the case of E.coli,  it is shown in \cite{KEVSC,TG} that by adding molecular noise to the signally pathway of the bacterium, one can also observe power-law switching in bacterial flagellar motors. Furthermore, the model in \cite{Matt09} suggests that fluctuation in CheR (a protein which regulates the receptor activity) can lead to a heavy-tailed distribution of run duration.

 In order to explain the aforementioned experimental and theoretical observations, in \cite{PST} limiting fractional diffusion equations are derived by adding noise into the pathway-based kinetic model~\eqref{eq:model-no-noise}. Although this work confirms that strong noise and slow adaptation can induce L{\'e}vy-flight type movement, it cannot be compared to the experimental observations in a qualitatively way. The main reason is that some assumptions in \cite{PST} are solely for mathematical convenience which can hardly be biologically relevant. In addition, they make the numerical simulation difficult if not impossible. For example, the adaptation function $f(\cdot)$ and the tumbling frequency $\Lambda(\cdot)$ are chosen for the purpose of analysis instead of observing their biological origin as in \eqref{eq:fadef} and \eqref{eq:Lambdaa}. Moreover, the diffusion coefficient (or the noise) in~\cite{PST} will tend to infinity in the region important for the generation of fractional diffusion.

Our main contribution of the current paper is to overcome the drawbacks described above. In particular, we start from an individual-based model (IBM), which incorporates a description of intracellular signaling, with the noise and the adaptation both bounded. We perform numerical simulations using the IBM and investigate population level behavior in different parameter regimes. The mesoscopic model associated with the IBM is a non-classical kinetic equation, from which we rigorously derive a limiting fractional diffusion equation (using a similar method as in~\cite{PST}). Parameters of the kinetic model are fully determined by those of the IBM. In addition, we will apply the particular formulas for the adaptation function and the tumbling frequency in~\eqref{eq:fadef} and~\eqref{eq:Lambdaa} to make possible of experimental verifications.


For completeness, in this paper we also include the rigorous justification of the fractional diffusion limit from the kinetic equation. The argument follows a similar line as in~\cite{PST}. It is now well-known that fractional diffusion limits can be derived from kinetic models. For a more extensive review we refer the read to~\cite{PST}.
Here we make one remark regarding kinetic equations with extended variables, that is, with variables in addition to the classical $(t, x, v)$. There has been several works showing fractional diffusion limits of extended kinetic equations~\cite{FS, EGP2018,EGP2019}. 
The models considered in \cite{FS, EGP2018,EGP2019} explicitly involve path length distributions or resting time distributions with power-law decay, while in our model we do not have these pre-set distributions. Instead we start with the biological signally pathway of bacteria. Thus our analysis can be viewed as an explanation of the origin of the power-law decay distributions. Such connection is made explicit in Section~\ref{sec:numerics} where we numerically show a comparison between the decay rate of the underlying path-length distribution and the power of fractional diffusion.

The paper is organized as follows. In Section~\ref{sec:model} we set up both the individual-base model (IBM) and the associated kinetic equation. Section~\ref{sec:numerics} is devoted to the numerical simulation of the IBM in different parameter regimes. Under proper scaling and conditions on the tumbling frequency as well as the form of noise, L\'{e}vy-type motions are observed numerically at the population level. In Section~\ref{sec:theory}, we prove the main theorem regarding the derivation of the fractional diffusion equation from the kinetic equation. The numerical results of IBM and theoretical derived fractional power are consistent.


\section{Models} \label{sec:model}
In this section we introduce two mathematical models: the individual-based model (IBM) and kinetic PDE model. Parameters in the kinetic model are fully determined by those in the IBM.
\subsection{Individual-based model} 
In the IBM, each cell is described as a particle with position $x^i$, velocity $v^i$ and activity $a^i$. The superscript $i$ is the index for the cell. By the definition of $a$ in \eqref{eq:adef}, we have
\beq\label{eq:papm}
\frac{\p a}{\p m}=N\alpha a(1-a).
\eeq
When there is no noise, the time evolution of $a^i$ is modeled by rewriting \eqref{eq:fadef} into an ODE for the activity ~$a$:
$$
\frac{\da}{\dt}=F(a)=k_R\Big(N\alpha a(1-a)\Big)(1-a/a_0).
$$
With noise, the activity $a^i$ can be modelled by a stochastic differential equation (SDE), which writes
\beq\label{eq:SDE-a}
	\da = F(a)\dt +\Sigma(a)\dBt,
\eeq
where $\{ B_t\}$ denotes a Wiener process, $F(a)$ is the adaptation and $\Sigma(a)$ is the strength of the noise.

The average tumbling time for E.coli cells is about 10 times shorter than their average running time. Thus we ignore the tumbling time at each turning and assume that the rate for a running bacterium going through the process of stopping, choosing a new direction and running again is $\Lambda$. We further assume that the new direction chosen by the bacterium is random with uniform distribution. The tumbling rate $\Lambda$ depends on the activity. For the $i$-th cell with activity $a^i$, the tumbling rate is given by
\beq\label{eq:Lambdaai}
\Lambda(a^i)=\tau^{-1}(a^i/a_0)^H.
\eeq
 Compared with \eqref{eq:Lambdaa}, the rotational diffusion $\lambda_0$ has been ignored in \eqref{eq:Lambdaai}. This is because we trace the bacteria trajectories and record their actual run lengths instead of the Euclidean distance between two successive tumbles. 
Analytically, the degeneracy of $\Lambda$ when $a^{i}$ vanishes is the key to generate long jumps. 
  
  
\begin{rmk}
 Another way of adding noise to the signally pathway is to use the methylation level $m$:
\beq\label{eq:SDE-m}
\dm=f(a)\dt+\sigma(a)\dBt
\eeq  
When $m$ is determined, the activity $a$ can be updated by the relation between $a$ and $m$ in \eqref{eq:adef}. Both formulations are reasonable when the ligand concentration $[L]$ is uniform in space, but when $[L]$ depends on space or time, the activity $a$ can no longer be considered as solely depending on $m$. Thus biophysicists prefer the intracellular pathway model in \eqref{eq:SDE-a}. Moreover, an important advantage of choosing the activity $a$ as the internal variable is that the boundedness of the internal variable $a$ (and the strength of the noises) make possible of numerical simulations for the model equation \eqref{eq:SDE-a}.
 \end{rmk}

\subsection{The PDE model}
The kinetic model associated with~\eqref{eq:SDE-a} describes the time evolution of the probability density function of the bacteria at time $t$, position $x\in\mathbb{R}^d$, velocity $v\in\VV$ 
and activity $a\in [0,1]$. 
By It\^{o}'s formula, the Fokker-Planck operator associated with~\eqref{eq:SDE-a} is 
\beq\label{eq:SDE-FPE}
\CalL q=-\partial_a(F(a)q)+\frac{1}{2}\partial_{aa}\big(\Sigma^2(a)q\big).
\eeq 
The null space of $\CalL$ is given by 
\begin{align*}
    \NullL = \Span \{Q_0\},
\end{align*}
where $Q_0$ satisfies
\begin{align} \label{eq:Q-0-1}
FQ_0=\frac{1}{2}\partial_a\big(\Sigma^2Q_0\big)
\end{align}
and can be solved explicitly as
\beq\label{eq:Q0a}
Q_0(a)=\frac{1}{c_0}\exp\Big(\int_{a_0}^a \frac{F\big(a'\big)-2\Sigma\big(a'\big)\partial_a\Sigma(a')}{\Sigma^2(a')}\da'\Big).
\eeq
Here $c_0$ is the normalization factor to make $\int_0^1Q_0\da=1$. 
Using such $Q_0$ one can rewrite $\CalL$ as
\begin{align}\label{eq:pa}
\CalL q
=\partial_a\Big(\frac{1}{2}\Sigma^2 Q_0\frac{-Fq+\frac{1}{2}\partial_a(\Sigma^2 q) }{\frac{1}{2}\Sigma^2Q_0}\Big)
=\partial_a\Big(\frac{1}{2}\Sigma^2Q_0\partial_a\big(\frac{q}{Q_0}\big)\Big).
\end{align}
Denote the diffusion coefficient as $D$ such that
\beq
\label{eq:DSigma} 
D(a)=\frac{1}{2}\Sigma^2(a).
\eeq 
Then the kinetic PDE model we consider has the form
\begin{align} \label{eq:kinetic-a}
   \del_t q
+  v \cdot  \nabla_x q
-  \del_a &\vpran{D(a) Q_0(a) \del_a\frac {q}{Q_0} }
= \Lambda(a) (\vint{q} - q) \,,
\end{align} 
Function $Q_0(a)$ can be viewed as the equilibrium distribution in $a$ in the absence of any external signal. 
The individual-based model can be considered as a Monte Carlo particle simulation for the kinetic PDE model \eqref{eq:kinetic-a}.  We can recover $F, D$ in the individual-based model from~\eqref{eq:Q-0-1} and~\eqref{eq:DSigma} by
\begin{align*}
F(a)=D(a)\frac{\del_aQ_0}{Q_0}+\del_a D,\qquad \Sigma(a)=\sqrt{2D(a)}.
\end{align*}


\begin{rmk} We can also write equation \eqref{eq:model-no-noise} in terms of $t$, $x$,  $v$ and $a \in [0, 1]$. Let $\tilde p(t,x,v,a)=p(t,x,v,m)$. We can deduce that 
$$
\del_t \tilde p
+ v \cdot \nabla_x \tilde p
+ N\alpha a(1-a)\del_a \vpran{f(a)\tilde p}
= \Lambda(a) (\vint{\tilde p} - \tilde p) \,.
$$
Compared with~\eqref{eq:kinetic-a}, we have $q(t,x,v,a)=\frac{\tilde p}{N\alpha a(1-a)}$.
\end{rmk}

\section{Numerical Results of the IBM}\label{sec:numerics}

\subsection{Scalings}
We nondimensionalize \eqref{eq:kinetic-a} by letting
\begin{align*}
   t = T_t \, \tilde{t},
\qquad 
   x = L \, \tilde{x},
\qquad 
   v_0 = V_0 \, \tilde{v}_0,
\qquad
   D(a)=\frac{\tilde D(a)}{T_a},
\qquad
   \Lambda = \frac{\tilde{ \Lambda}}{T_\lambda},
\end{align*}
where $T_t$, $L$ and $V_0$ are respectively the characteristic temporal, spatial and velocity scales of the system. The parameters $T_a$ and $T_\lambda$ are the characteristic adaptation time and running time between two successive tumbles. 
The nondimensionalized equation is (after dropping ``$\sim$"): 
\begin{align} \label{eq:nondimensional-kinetic-a}
   \frac{T_\lambda}{T_t} \del_t q_\Eps
+\frac{V_0 T_\lambda}{L}  v \cdot  \nabla_x q_\Eps
- \frac{T_\lambda}{T_a} \del_a &\vpran{D(a) Q_0(a) \del_a\frac {q_\Eps}{Q_0} }
= \Lambda(a) (\vint{q_\Eps} - q_\Eps) \,,
\end{align}
We consider the scaling such that 
\beq\label{eq:scales}
    \frac{T_\lambda}{T_t}=\Eps^{1+\mu},
\qquad
   \frac{V_0 T_\lambda}{L}=\Eps,\qquad \frac{T_\lambda}{T_a}=\Eps^s,
\eeq
where $\mu\in[0,1]$ and $s<1+\mu$ will be determined later. Different $s$ are tested numerically and the magnitude of $s$ gives the time scale of the intracellular signal dynamics.

\subsection{Simulations of the IBM}
We use the particle method in one space dimension to verify that in certain parameter regimes, one can observe a L{\'e}vy-flight type movement instead of the Brownian motion on the population level. 
  
\medskip
\Ni {\textbf{Numerical scheme.}}
The computational domain is $[-25mm,25mm]$ and we track the trajectory of $10000$ cells. Each cell is represented by its position $x^i$, velocity $v^i$ and activity $a^i$. The initial $x^i$ for all cells are $0$, their initial velocities $v^i$ are randomly set to be $v_0$ or $-v_0$ with equal probability  and the initial $a^i$ for the $10^4$ particles are randomly distributed according to $Q_0(a)$. Let $\Delta t$ be the time step. At each step we evolve $(x^i,a^i,v^i)$ ($i=1,\cdots,10^4$) by the following calculations: 

\begin{itemize}
\item[1)] Update $v^i$ according to the tumbling frequency $\Lambda(a^i)$.
For each $i$, generate one random number $r^i$ uniformly distributed in $[0,1]$. If $r^i\leq\frac{1}{2} \Lambda(a^i)\Delta t$, then set the cell velocity to $-v^i$.
\item[2)] Update the position $x^i$.
Set the new position to be $x^i+v^i\Delta t$ with $v^i$ being the current cell velocity.

\item[3)] Update the internal state $a^i$.
We update $a^i$ according to the SDE in \eqref{eq:SDE-a} by the widely-used Milstein Scheme \cite{}. The new $a^i$ is set to be
\begin{align*}
a^i+F(a^i)\Delta t+\Sigma(a^i)\Delta B_t+\frac{1}{2}\Sigma(a^i)\partial_a\Sigma(a^i)\big((\Delta B_t)^2-\Delta t\big),
\end{align*}
where $\Delta B_t$ is a random number according to the normal distribution $N(0,\Delta t)$.
\end{itemize}

\Ni{\textbf{Results.}}
In the physics literature, the run length distribution and the relation between the time and mean square displacement (MSD) are used to determine whether a L{\'e}vy-flight type movement occurs. The MSD is defined by:
\begin{align*}
\text{MSD} := \vint{(\text{x}-\text{x}_0)^2}=\frac{1}{N}\sum\limits_{i=1}^{N}(x^i(t)-x^i(0))^2,
\end{align*} 
where $N$ is the number of particles, $x^i(0)$ is the initial position of the $i$th particle and $x^i(t)$ is the position of the $i$th particle at time $t$.  We run simulations with $10000$ cells and record all run lengths between two successive velocity switching events as well as their MSD. 
The MSD satisfying $\text{MSD}\sim t^{\frac{2}{1+\mu}}$ corresponds to a L{\'e}vy-walk with the fractional power $1+\mu$, while for a Brownian motions one should have $\text{MSD}\sim t$.


Most of the parameters we choose are from the wild time E.coli. In the definition of $a$ and $\Lambda(a)$, 
we take $$N=6,\qquad \alpha=1.7,\qquad a_0=1/2,\qquad \tau=1 \, \text{sec},\qquad \Lambda(a)=(a/a_0)^\beta$$ as in \cite{SWOT}. By tuning the adaption time as well as the noise, we can observe different population level behaviour. More specifically, we choose
\begin{align*}
F(a)=\frac{n+\theta}{T_a}\big(a(1-a)\big)^{n-1}(1-2a),\qquad
\Sigma(a)=\sqrt{\frac{2}{T_a}}\big(a(1-a)\big)^{n/2},\qquad
\Lambda(a)=(2a)^\beta.
\end{align*}
From \eqref{eq:Q0a} and \eqref{eq:DSigma}, the corresponding $Q_0(a)$ and $D(a)$ are
\begin{align*}
Q_0(a)=\big(a(1-a)\big)^{\theta},\qquad
D(a)=\frac{1}{T_a}\big(a(1-a)\big)^{n}.
\end{align*}
Throughout the computation, we fix
\begin{align*}
v_0 = 0.02 \, \text{mm/sec},
\qquad L=1 \, \text{mm}, 
\qquad T_\lambda=1 \, \text{sec}.
\end{align*}
This gives 
\begin{align*}
    \Eps=0.02, \qquad \Eps^s=1/T_a.
\end{align*}
Thus different $s$ corresponds to different adaptation time $T_a$. 
According to the theoretical results in Section 2.3, long jumps happen when $s\in(1+\frac{\theta+1}{\beta},1+\frac{2-n}{\beta})$, . Then if we choose 
\begin{align*}
\theta=-0.5,\qquad n=1.1,\qquad \beta=2,
\end{align*}
then $s\in(1.25,1.45)$ and $T_a\in(132.96 \,\text{sec}, \, 290.74 \, \text{sec})$. The other important parameter is the characteristic system time $T_t$. From the main theorem in Section 2.3, the system time should be $\Eps^{-1-\frac{2-n}{\beta}}$ which is approximately $300 \, \text{sec}$.
In what follows, we test three different sets of parameters:
\begin{itemize}
\item[I]$\theta=-0.5$, $n=1.1$, $\beta=2$, $T_a=200s$, $T_t=300s$;
\item[II]$\theta=-0.5$, $n=1.1$, $\beta=2$, $T_a=10s$, $T_t=300s$;
\item[III]
$\theta=0.5$, $n=1.1$, $\beta=2$, $T_a=200s$, $T_t=300s$.
\end{itemize}
Among the three only Case I satisfies the conditions for long jumps. Case II has a fast adaptation and the equilibrium $Q_0$ in Case III does not satisfy the constraints in~\eqref{assump:main}. We expect the classical diffusion occurs in Case II and possibly in Case III as well. 

The time step we use is $\Delta t=0.1$ sec. The simulation is run up to $T_t$. We record all path lengths $\{l_j\}$ between two successive velocity switching events for all $10000$ particles. Since the cell velocity is $0.02$ mm/sec and $\Delta t=0.1$ sec, the smallest possible run length is $0.002$ mm. Let $\Delta x=0.002$ and 
$$
x_i=(i+0.5)\Delta x,\qquad\mbox{for $i\in\{0,1,\cdots,[10T_t]\}$}.
$$
We count the number of $l_j$'s that fall into the interval $(x_i-\frac{\Delta x}{2},x_i+\frac{\Delta x}{2}]$ and denote them by $y_i$. In Figure~ \ref{fig:loglog}, we plot the collection of $(\log(l_i),\log(y_i))$ and $(l_i,\log(y_i))$. According to \cite{FS}, the decay rate of the path length distribution (PLD) plays an important role in determine whether a random process is Brownian motion or L{\'e}vy-flight type movement.  More specifically, when the PLD decays as $s^{-\alpha}$ with $\alpha\in(2,3)$, the cells exhibit L{\'e}vy-flight type movement at the population level. When the PLD decays as $s^{-\alpha}$ with $\alpha\geq 3$ or $\exp(-\alpha s)$ for arbitrary positive $\alpha$, the cells follow a Brownian motion. Therefore, we investigate that if $(\log(l_i),log(y_i))$ or $(l_i,\log(y_i))$ can be fitted by a linear function when $y_i$ is small. However, 
 as we can see from the data, many $l_i$ share the same $y_i$, especially when $y_i$ is small. Therefore, we average all $l_i$ with the same $y_i$ and then use the least square method to find the straight lines that fit best in the tail part where $y_i\in[1,1000]$.

The path lengths, their frequency and the fitted straight lines are presented in Figure \ref{fig:loglog}. We can see that the PLD does obey a power law decay in the expected range in Case I and an exponential decay in Case II. 
However, in Case III, although we give one possible fit, the data is actually too noisy for deciding into which range its decay rate falls.

The situation is more clear in Figure~\ref{fig:MSD}, where the MSD at different time for different cases are plotted. Figure \ref{fig:MSD} confirms that Case I corresponds to long jumps and Case II to Brownian motions. The slop in Case III suggests that it is a normal diffusion. We further note that the slope in Figure~\ref{fig:MSD} for Case I corresponds to a fractional power of $2/1.382 \approx 1.447$, while the theoretical result gives $1+\mu = 1 + \frac{2-1.1}{2} = 1.45$.

Another observation we make is that if we run the simulation for Case I for a longer time until $T_t=5000 \, \text{sec}$, then the slope change from $1.382$ to $1.0423$. This is consistent with some experimental observations that fractional diffusions can evolve into the normal diffusion when the time is long enough \cite{Hongliang}. 
\begin{figure}[htbp!]\centering
\caption{The run lengths of $10000$ cells. The top figures are the log-log plots of different cases. The bottom one are the log plots of different cases.}
\includegraphics[scale=0.75]{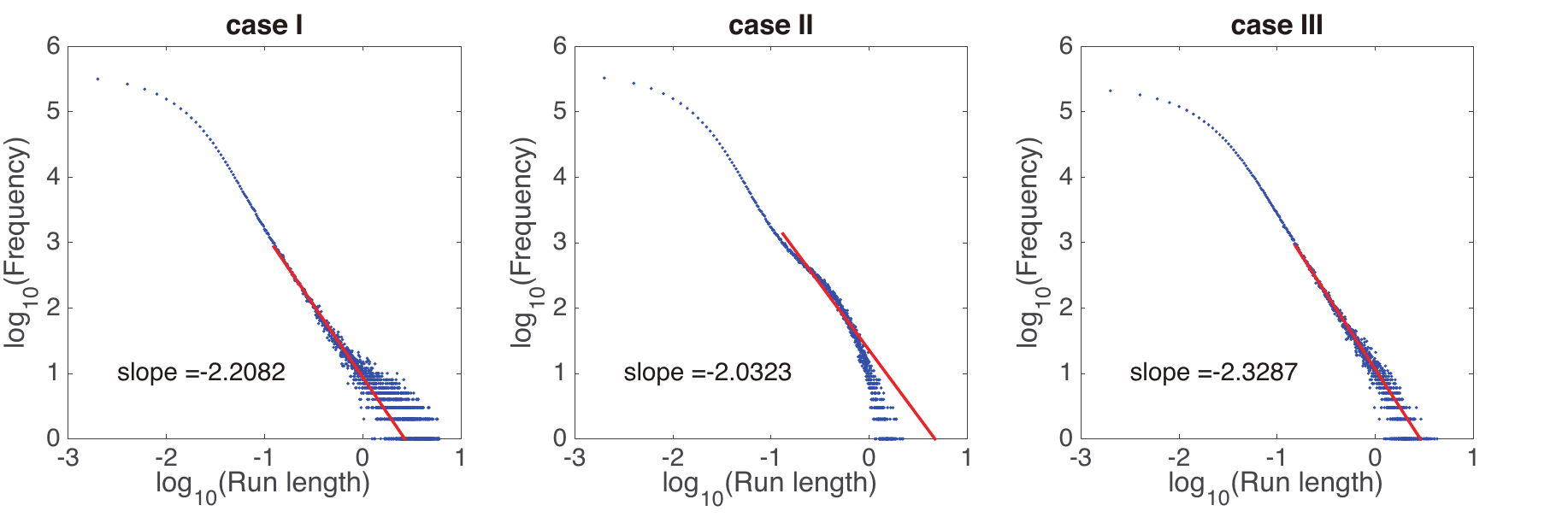} 
\\
\hspace{0.15cm} \includegraphics[scale=0.75]{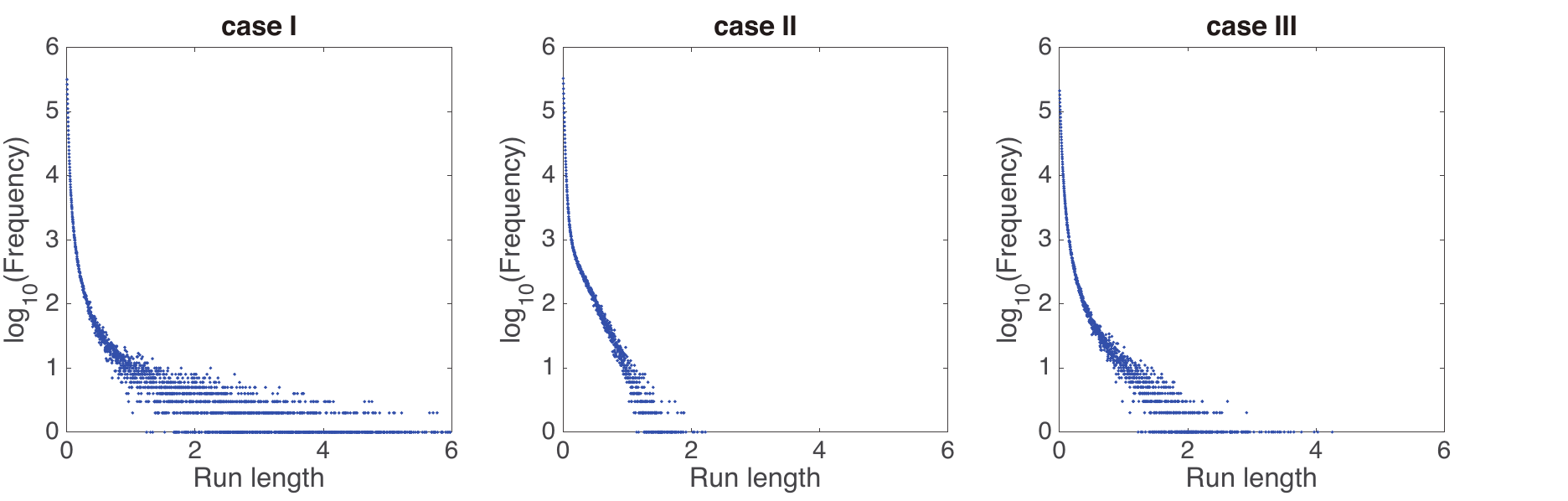}
\label{fig:loglog}
\end{figure}

\begin{figure}[htbp!]\centering
\caption{The mean square displacement of $10000$ cells for different time. The slope of the MSD in Case I corresponds to a fractional power of $1.447$, which is in quantitative agreement with the theoretical result of $1.45$.} 

\medskip
\includegraphics[scale=0.55]{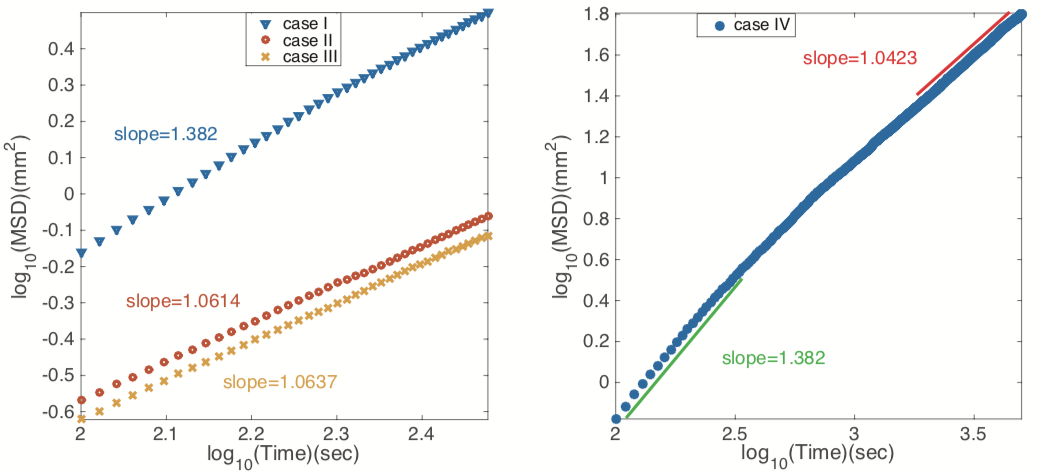}
\label{fig:MSD}
\end{figure}


\section{Theoretical results based on the PDE model}\label{sec:theory}
The main theoretical result of this paper is to rigorously derive fractional diffusion equations (which correspond to L{\'e}vy processes) from the following scaled equation 
\begin{align} \label{eq:scaled-kinetic-a}
   \Eps^{1+\mu} \del_t q_\Eps
+ \Eps v \cdot  \nabla_x q_\Eps
- \Eps^s \del_a &\vpran{D(a) Q_0(a) \del_a\frac {q_\Eps}{Q_0} }
= \Lambda(a) (\vint{q_\Eps} - q_\Eps) \,,
\end{align}
where $0 < \mu < 1$ and $0< s < 1+ \mu$. 

\smallskip

\Ni {\bf Assumptions on the coefficients. }Let $a_1\in(0,1/2)$ and $a_2\in(1/2,1)$ be two constants.
\begin{itemize}
\item
The equilibrium state $Q_0(a)$ satisfies 
\begin{align} \label{def:q0}
     Q_0(a) 
     =  \begin{cases}
           c'_q a^{\theta'}\,, & a > a_2 \,, \\[2pt]
           \BigO(1) \,, & a \in [a_1, a_2] \,, \\[2pt]
          c_q a^{\theta} \,, & a < a_1 \,,
         \end{cases}
        \qquad    Q_0(a) >0, \qquad   \int_0^1 Q_0\,da=1 .
\end{align}
\item The tumbling frequency $\Lambda(a)$ has the structure that $\Lambda \in C^1[0, a]$ and
\begin{align} \label{def:Lambda}
   \Lambda(a) 
   = \begin{cases} 
        \Lambda_0(a)  \,, & a \geq a_1 \,, \\[2pt]
        c_\lambda a^\beta \,, & a \leq a_1 \,,
      \end{cases} 
\qquad
   \Lambda_0(a) \geq \lambda_0 > 0 \,,
\end{align}
where $\lambda_0$ is a constant. 
The mechanism at work here is the degeneracy of  the tumbling rate $\Lambda(a)$ near $a = 0$.
\item
The diffusion coefficient $D(a)$ is a smooth bounded functions on $[0,1]$ such that
\begin{align} \label{def:D}
D(a) 
     =  \begin{cases}
          c'_d a^{n'} \,, & a > a_2 \,, \\[2pt]
          \BigO(1) \,, & a \in [a_1, a_2] \,,  \\[2pt]
          c_d a^{n} \,, & a < a_1 \,,
         \end{cases}
\end{align}
\end{itemize}
The parameters $\theta$, $\theta'$, $c_q$, $c_q'$, $\beta$, $c_\lambda$, $n$, $n'$, $c_d$, $c_d'$ are all positive constants. \smallskip

 \Ni {\bf Assumptions on the initial data.} We assume that
 \begin{align*}
q (0, x, v, a)  = q^{in}(x, v, a) := \rho^0(x) Q_0(a) \geq 0 \,.
\end{align*}
and there exists a constant $B$ such that
\begin{align}\label{ID}
q^{in} \leq B Q_0, 
\qquad      
   \int_{\R^d} \int_0^1 \int_\VV \frac{(q^{in})^2}{Q_0} ( x, v, a) \dv \da \dx  \leq B, 
\qquad   
   \int_{\R^d} \int_0^1 \int_\VV q^{in} ( x, v, a) \dv \da \dx \leq B \,.
\end{align}
The main theorem that we prove in this paper states 
 \begin{thm*} 
Let $q_\Eps$ be the solution of~\eqref{eq:scaled-kinetic-a}. Suppose assumptions \eqref{def:q0}--\eqref{ID} hold and the parameters $\theta,\theta', \beta, n, n', \mu$ satisfy 
\begin{align}\label{assump:main}
   -1<\theta<1 - n < \beta - 1 \,,
\qquad 
   1+\frac{\theta+1}{\beta}<s<1+\mu \,,\
\qquad 
   -1 < \theta' < 1 - n'.
\end{align}
Moreover, suppose that
\begin{equation} \label{value_mu}
\mu = \frac{2-n}{\beta} \in (0, 1) \,.
\end{equation} 
Then we have
\begin{equation}
q_\Eps (t,x,v,a)  \to \rho(t,x) Q_0(a)
\qquad
 \text{as $\Eps \to 0$}
\label{limi_gen}
\end{equation} 
in the sense that $\frac{q_\Eps}{Q_0}$ converges  $L^\infty-w*$ to  $\rho \in L^\infty(\R^+; L^1\cap L^\infty(\R^d))$ and $\rho$ solves
\begin{equation} 
\left\{ \begin{array}{l}
 \del_t  \rho(t,x) + \nu  \vpran{-\Delta}^{\frac{1+\mu}{2}}\rho=0,
\\ 
 \rho(0,x) = 
 \rho^0(x) \,.
\end{array}\right. 
\label{eq:fractional}
\end{equation}
The diffusion coefficient is given by $\nu = B_0/\nu_0$ where $B_0, \nu_0$ are defined in ~\eqref{def:B-0} and~\eqref{def:nu-0} respectively.
\label{thm:main}
\end{thm*}

The main theorem shows that, within a certain parameter regime, the population level behaviour can adopt a L{\'e}vy-flight type movement if there is noise in the internal signally pathway. This phenomenon appears if the tumbling frequency $\Lambda $ has degeneracy. If instead $\Lambda$ has a strictly positive lower bound, then a classical diffusion will occur. In proving the fractional diffusion limit in~\eqref{eq:fractional}, we use the same techniques as in \cite{PST} with a particular attention paid to the singularity at $a=0$.



We will use the notations
\begin{align} \label{def:rho-R-Eps}
  \rho_\Eps(t,x)= \int_0^1 \int_\VV q_\Eps(t,x,a) \dv\da, 
\qquad
   R_\Eps (t, x, v) = \int_0^1 q_\Eps(t,x,a)\da.
\end{align}
\subsection{Useful bounds}\label{subsec:apriori} Similar as in~\cite{PST}, we first derive several useful bounds and establish some technical lemmas.

\subsubsection{Relative Entropy Estimates} Since equation~\eqref{eq:scaled-kinetic-a} has a similar structure with equation (0.2) in~\cite{PST}, we have the following relative entropy estimate:
\begin{lem} \label{lem:entropy}
Supppse $q_\Eps$ is a solution to equation~\eqref{eq:scaled-kinetic-a}. Suppose the initial data $q^{in}$ satisfies that 
\begin{align*}
   \int_{\R^d} \int_{\R} \int_\VV q^{in} \dv\da\dx = 1
\quad \text{and} \quad
    0 \leq q^{in} \leq B Q_0
\quad
   \text{for some constant $B > 0$.}
\end{align*}
Then $q_\Eps$ satisfies that $ 0\leq q_\Eps \leq B Q_0$ and
\begin{align}\label{L2-infty}
     0\leq q_\Eps \leq B Q_0 \,,
\qquad
  \int_{\R^d} \int_0^1 \int_\VV
     \frac{q_\Eps^2}{Q_0} (t, x, v, a) \dv \da \dx \leq B,  
\end{align}
and
\begin{align} \label{bound:orthogonal}
  \int_0^\infty \int_{\R^d} \int_0^1 \int_\VV
     D(a) Q_0(a) \vpran{\del_a \vpran{\frac{q_\Eps}{Q_0}}}^2
\leq B \Eps^{1+\mu-s} \,,
\qquad
   \int_0^\infty \int_{\R^d} \int_0^1 \int_\VV 
      \Lambda(a) \frac{\vpran{q_\Eps - \vint{q_\Eps}}^2}{Q_0}
\leq 
  B  \Eps^{1+\mu}  \,.
\end{align}
\end{lem}
The details of the proof of Lemma~\ref{lem:entropy} are omitted since they are the same as the proof of Lemma 2.1 in~\cite{PST} (with the variable $y$ in~\cite{PST} changed into $a$ here). 

The first and immediate consequence of Lemma~\ref{lem:entropy} is the weak convergence of $q_\Eps$:
\begin{lem} \label{lem:wcv} 
Along a subsequence still denoted by $q_\Eps$, we have
$$
 \f{q_\Eps }{Q_0} (t, x, v, a) \to \rho(t,x) , \qquad \text{in } L^\infty(\R^+ \times \R^d \times \R \times \VV)-w^\ast \, ,
$$
where $\rho(t,x)\in   L^\infty (\R^+ ; L^1\cap L^\infty (\R^d))$. 
\end{lem}

\subsubsection{A priori bounds}
The following a priori bound is similar to Lemma 2.3 in~\cite{PST}. However, since the singularity now appears at $a=0$, for the convenience of the reader we show the full proof of the lemma.
\begin{lem} \label{lem:a-priori}
Suppose $q_\Eps$ satisfies  the a priori bound~\eqref{bound:orthogonal}. 
Then there exists a constant $C > 0$ independent of $t, x, a$ and $\Eps$ such that
\begin{align} \label{bound:deviation}
 \abs{ \f{q_\Eps }{Q_0} (t, x, v, a) - R_\Eps (t, x,v)}
\leq C H^{1/2}(t, x,v)  \,, \qquad \forall a \in [0,1],\, v\in\VV, 
\end{align}
where $R_\Eps$ is defined in~\eqref{def:rho-R-Eps} and
\begin{align} \label{def:H}
H(t, x,v) = \int_0^1 Q_0(a')D(a')\vpran{\p_{a'}\Big(\f{q_\Eps(t, x, v, a')}{Q_0(a')}\Big)}^2\da'  \,.
\end{align}
\end{lem}
\begin{proof}
By the a priori bound~\eqref{bound:orthogonal}, it holds that
\begin{align*}
 \abs{ \f{q_\Eps}{Q_0} - R_\Eps}
&=  \abs{\f{q_\Eps(a)}{Q_0(a)}- \int\f{q_\Eps(a')}{Q_0(a')} Q_0(a')\da'}
\leq 
     \int_0^1 \abs{\f{q_\Eps(a)}{Q_0(a)}-\f{q_\Eps(a')}{Q_0(a')} }Q_0(a')\da'
\\
&\leq  \int_0^1 \vpran{\int_{a'}^a \abs{\p_{z} \vpran{\f{q_\Eps(z)}{Q_0(z)}}}\dz} Q_0(a')\da'
\\
 &\leq 
       \int_0^1 \vpran{\int_{a'}^a Q_0(z)D(z)\vpran{\p_{z}
                              \vpran{\f{q_\Eps(z)}{Q_0(z)}}}^2\dz}^{1/2}
                    \vpran{\abs{\int_{a'}^a\f{1}{Q_0(z)D(z)}\dz}}^{1/2} Q_0(a') \da'
\\
&\leq 
     \vpran{ \int_0^1 Q_0(a') \vpran{\abs{\int^1_0 \frac{1}{Q_0(z)D(z)}\dz}}^{1/2}\da'}   H^{1/2}(t, x,v) \,.
\end{align*}
Since $\theta + n < 1$ and $\theta' + n' < 1$, we have
\begin{align*}
   \int^1_0 \frac{1}{Q_0(z)D(z)}\dz < \infty \,.
\end{align*}
Hence~\eqref{bound:deviation} holds with $C = \vpran{\abs{\int^1_0 \frac{1}{Q_0(z)D(z)}\dz}}^{1/2}$.
\end{proof}

Denote the Fourier transform in $x$ of $q$ as $\hat q$, which is defined by
\begin{align*}
 \hat q (t, \xi, v, a) =\int_{\R^d} q(t,x,v,a)e^{ix.\xi} dx.
\end{align*}
Then by the using Parseval identity, the Fourier version of the above apriori estimates are 
\begin{lem} \label{lem:Fourier-bound}
Let $q_\Eps$ be the solution to~\eqref{eq:scaled-kinetic-a}. Then
\begin{align} \label{bound:fourierL2v}
  \int_{\R^d} \int_0^1 \int_\VV
     \frac{q_\Eps^2}{Q_0} (t, x, v, a) \dv \da \dx \leq B,  
\qquad
 \int_0^\infty \int_{\R^d} \int_0^1 \int_\VV 
       \Lambda(a) \frac{\abs{\hat q_\Eps - \vint{\hat q_\Eps}}^2}{Q_0} \dv\da\dxi\dt
\leq 
  B  \Eps^{1+\mu}  \,.
\end{align}
Moreover, if we denote 
\begin{align} \label{def:K}
K(t, \xi,v) = \int_0^1 Q_0(a)D(a)\abs{\p_{a}\Big(\f{\hat q_\Eps(t, \xi, a,v)}{Q_0(a)}\Big)}^2 \da  \,,
\end{align}
then
\begin{align} \label{est:Kfourier}
\int_0^\infty \int_\VV \int_{\R^d}  K(t, \xi ,v) \dxi \dv   \dt = \int_0^\infty \int_\VV \int_{\R^d}  H(t, x,v) \dx \dv   \dt  \leq  B  \Eps^{1+\mu-s}
\end{align}
and
\begin{align} \label{bound:fouriery}
\abs{ \f{\hat q_\Eps }{Q_0} (t, \xi, v, a) - \hat R_\Eps (t, \xi,v)}
\leq C K^{1/2}(t, \xi,v)  \,, \qquad \forall a \in \R,\, v\in\VV, 
\end{align}
\end{lem}

The fractional power will be derived by using the following lemma:
\begin{lem} \label{lem:integral}
Suppose 
\begin{align*}
   0 < \alpha - 1 < 2\beta_1 \,,
\qquad
  0 < \alpha - 1 <  4\beta_2 \,,
\qquad
  \beta_1, \beta_2 > 0 \,.
\end{align*}
Then the following integrals are well-defined and there exists a constant $c_1 > 0$ such that
 \begin{align*}
   \int_{0}^{a_1} \frac{|a|^{-\alpha}}{1 + \vpran{\Eps |\xi \cdot v|  |a|^{-\beta_1}}^2} \da
= c_1 \vpran{\Eps |\xi \cdot v| }^{-\frac{\alpha-1}{\beta_1}} \,,\qquad
\int_{0}^{a_1} \frac{|a|^{-\alpha}}{(1 + \vpran{\Eps |\xi\cdot v|  |a|^{-\beta_2}}^2)^2} \da
= c_2 \vpran{\Eps  |\xi\cdot v| }^{-\frac{\alpha-1}{\beta_2}} \,.
\end{align*}
\end{lem}
\begin{proof}
Make a change of variable $z = \Eps |\xi \cdot v|  |a|^{-\beta_1}$ in the first integral and $z = \Eps |\xi \cdot v|  |a|^{-\beta_2}$ in the second one. Then
\begin{align*}
   \int_{0}^{a_1} \frac{|a|^{-\alpha}}{1 + \vpran{\Eps |\xi \cdot v|  |a|^{-\beta_1}}^2} \da
= \frac{1}{\beta_1} \vpran{\Eps |\xi\cdot v| }^{-\frac{\alpha-1}{\beta_1}}
    \int_{\Eps |\xi\cdot v||a_1|^{-\beta_1}}^\infty \frac{z^{\frac{\alpha - 1}{\beta_1} - 1}}{1 + z^2} \dz
= c_1 \vpran{\Eps  |\xi\cdot v| }^{-\frac{\alpha-1}{\beta_1}} \,,
\end{align*}
\begin{align*}
   \int_{0}^{a_1} \frac{|a|^{-\alpha}}{(1 + \vpran{\Eps |\xi\cdot v|  |a|^{-\beta_2}}^2)^2} \da
= \frac{1}{\beta_2} \vpran{\Eps |\xi\cdot v| }^{-\frac{\alpha-1}{\beta_2}}
    \int_{\Eps |\xi\cdot v||a_1|^{-\beta_2}}^\infty \frac{z^{\frac{\alpha - 1}{\beta_2} - 1}}{(1 + z^2)^2} \dz
= c_2 \vpran{\Eps |\xi\cdot v| }^{-\frac{\alpha-1}{\beta_2}} \,,
\end{align*}
where the integrability of the $z$-integral is guaranteed respectively by the condition $0 < \frac{\alpha-1}{\beta_1} < 2$ and $0 < \frac{\alpha-1}{\beta_2} < 1$  , or equivalently, $0 < \alpha - 1 < 2 \beta_1$ and $0 < \alpha - 1 <  4\beta_2$.
\end{proof}

\subsubsection{Weight function}

We will use a weight function  built by duality.  In particular,
let $\chi_0(a)$ be given by 
\begin{align} \label{def:chi-0}
   \chi_0(a) 
   = \int_{0}^a \frac{1}{D(a')Q_0(a') } \da' \,.
\end{align}
Then it is a solution of the dual problem in $a$ because
\begin{align} \label{eq:chi-0}
\del_a\vpran{D(a) Q_0(a) \del_a\chi_0}=0.
\end{align}

Properties of $\chi_0$ follow immediately from the properties of $D, Q_0$ and they are summarized as
\begin{lem} \label{lem:chi-0}
With $Q, \, D$ as in~\eqref{def:q0}, \eqref{def:D} and with the parameter range~\eqref{assump:main}, $\chi_0 \in C[0,1]$ is nonnegative, increasing, bounded and 
\begin{align*}
   \chi_0 = C_0  a^{-\theta-n+1} \quad \text{for $a < a_0$}.
\end{align*}
\end{lem}

\subsection{Proof of Main Theorem} Now we are ready to show the proof of the main theorem. We start with the conservation law obtained via multiplying equation~\eqref{eq:scaled-kinetic-a} by the weight function $\chi_0(a)$ and integrating in $a$ and $v$. Thanks to the fact that $\chi_0$ solves the dual problem in~$y$, we find 
\begin{equation}\label{conservationLaw}
\partial_t \int_0^1 \int_\VV q_\Eps \chi_0 \da \dv +{\rm div}_x 
J_\Eps =0
\quad \text{with} \quad
J_\Eps=\f{1}{\Eps^{\mu}}  \int_0^1 \int_\VV v q_\Eps \chi_0 \da \dv .
\end{equation}
The limiting equation will be derived by the convergence in the distributional sense of the equation in~\eqref{conservationLaw}. 
By Lemma~\ref{lem:wcv}, the weak limit of the first term is 
\begin{align} \label{def:B-0}
 \int_0^1 \int_\VV q_\Eps \chi_0 \da \dv \to \int_0^1 \int_\VV \rho(t,x) Q_0(a) \chi_0 \da \dv = B_0 \rho(t,x), \qquad B_0  = \int_0^1 Q_0(a) \chi_0 \da .
\end{align}
It remains to identify the limit of the flux $J_\Eps$. Notice that the apriori estimates in Section~\ref{subsec:apriori} do not provide any direct $L^p$ bound on $J_\Eps$. To better understand the structure of $J_\Eps$ we resort to the Fourier method. Our eventual goal is to prove that, for the constant $\nu_0$ defined in~\eqref{def:nu-0}, as $\Eps \to 0$, 
\begin{equation}
   \hat{{\rm div}_ x J_\Eps} \to  \nu_0 |\xi |^{\f{2-n}{\beta} + 1}   \hat \rho \,,  
\qquad
\text{in the sense of distributions (or in $\CalD'(\R^+ \times \R^d)$)},
\label{est:RestL2}
\end{equation}
which will conclude the proof of the main theorem.

Apply the Fourier transform in $x$ to \eqref{eq:scaled-kinetic-a}, and denote by $\xi$ the Fourier variable. We obtain
 \begin{align*}
   \Eps^{1+\mu} \del_t \hat q_\Eps
+ i\Eps\xi \cdot v \,  \hat q_\Eps
 -\Eps^s \del_a\vpran{D(a)Q_0(a)\del_a \f{\hat q_\epsilon}{Q_0(a)}}
= \Lambda(a) (\vint{\hat q_\Eps} - \hat q_\Eps).
\end{align*}
Rearranging terms, we get
\begin{align}\label{eq:hatqeps}
   \hat q_\Eps - \vint{\hat q_\Eps}
=  - \frac{i \Eps \xi \cdot v}{i\epsilon\xi \cdot v+\Lambda} \vint{\hat q_\Eps} 
   + &\frac{\Eps^s}{i\epsilon\xi \cdot v+\Lambda}\del_a\vpran{D(a)Q_0(a)\del_a \f{\hat q_\epsilon}{Q_0(a)}}
    -\Eps^{1+\mu}\frac{1}{i\epsilon\xi \cdot v+\Lambda} \del_t \hat q_\Eps  \,.
\end{align}
By symmetry and~\eqref{eq:hatqeps}, the Fourier form of the flux term ${\rm div}_x J_\Eps$ can be written accordingly as
\begin{align}
  \hat{{\rm div}_x J_\Eps}(t,\xi) 
  = \f{1}{\Eps^{\mu}}\int_0^1 \int_\VV \vpran{ i \xi \cdot v}  \chi_0 \vpran{\hat q_\Eps-  \vint{\hat q_\Eps}} \da \dv
  =  i\xi\cdot\hat J_\Eps^1 +   \hat J_\Eps^2 + \partial_t    \hat J_\Eps^3.
\label{eq:Jdecomp}
\end{align}
We show in the following that  $\hat J_\Eps^2, \hat J_\Eps^3$ vanish as $\Eps \to 0$ and the fractional Laplacian arises from the $\hat J_\Eps^1$-term.

First we treat the $\hat J_\Eps^1$-term. Separate the imaginary and real part such that
\begin{align*}
    \frac{i \Eps \xi \cdot v}{i\epsilon\xi \cdot v+\Lambda}
 = \frac{\vpran{\Eps \xi \cdot v}^2}{\vpran{\Eps \xi \cdot v}^2 + \Lambda^2}
    + \frac{\vpran{i \Eps \xi \cdot v} \Lambda}{\vpran{\Eps \xi \cdot v}^2 + \Lambda^2} \,.
\end{align*}
Using the symmetry of $\VV$, one can see that contribution from the real part above vanishes and we have
\begin{align*}
 \hat J_\Eps^1 (t,\xi) 
 = \f{-1}{\Eps^{\mu}} \int_\VV \int_0^1  v \chi_0 \frac{i \Eps \xi \cdot v}{i\epsilon\xi \cdot v+\Lambda} \vint{\hat q_\Eps} \da \dv  
 = \f{-i}{\Eps^{\mu}} \int_\VV \int_0^1  v \chi_0 \frac{\Lambda \Eps \xi \cdot v}{(\epsilon\xi \cdot v)^2+\Lambda^2} \vint{\hat q_\Eps} \da \dv.
\end{align*}
Therefore we may write 
\begin{align} \label{def:hat-J-Eps-1}
 i \xi \cdot \hat J_\Eps^1 
 = \hat \rho_\Eps \f{1}{\Eps^{\mu}} \int_\VV \int_0^1  \chi_0 \frac{\Lambda \Eps \vpran{\xi \cdot v}^2}{(\epsilon \xi \cdot v)^2+\Lambda^2} Q_0(a) \da \dv 
 + i \xi \cdot \hat {RJ_\Eps^1} \,,
\end{align}
where the remainder term is
\begin{align} \label{def:RJ-Eps}
\hat {RJ_\Eps^1}  = \f{-i}{\Eps^{\mu}}  \int _\VV  \int_0^1 v \chi_0  \frac{\Lambda \Eps \xi \cdot v}{(\epsilon \xi \cdot v)^2+\Lambda^2} Q_0(a) \vpran{\f{ \vint{ \hat q_\Eps }}{ Q_0(a)}-  \hat  \rho_\Eps} \da \dv .
\end{align}

In order to derive the limit for the first term on the right-hand side of~\eqref{def:hat-J-Eps-1}, we divide the integration domain for $a$ into two parts: $a \geq a_1$ and $a < a_1$. Note that by the definition of $\Lambda$, we have $\Lambda \geq \lambda_0$ for $a \geq a_1$. Hence, the integral term in~\eqref{def:hat-J-Eps-1} over $a \geq a_1$ satisfies
\begin{align*}
    \abs{\f{-i}{\Eps^{\mu}} \int_\VV \int_{a > a_1}  v \chi_0 \frac{\Lambda \Eps \xi \cdot v}{(\epsilon \xi \cdot v)^2+\Lambda^2} Q_0(a) \da \dv}
\leq
   C \Eps^{1-\mu} |\xi| \int_0^1 \chi_0 Q_0 \da 
\leq 
   C \Eps^{1-\mu} |\xi| \,,
\end{align*}
As a consequence, 
\begin{align} \label{limit:J-Eps-1-1}
    \hat \rho_\Eps \f{-i}{\Eps^{\mu}} \int_\VV \int_{a > a_1}  v \chi_0 \frac{\Lambda \Eps \xi \cdot v}{(\epsilon \xi \cdot v)^2+\Lambda^2} Q_0(a) \da \dv
\to 0 
\qquad
   \text{in $\mathcal{D}'(\R^+ \times \R^d)$} \,.
\end{align}
The nontrivial contribution of the integral comes from the part $a \sim 0$ where $\Lambda(a)$ vanishes. By Lemma \ref{lem:integral}, the limit of this part is
\begin{align}
\hat{\rho} \lim_{\Eps \to 0} \f{1}{\Eps^{\mu}}   \int_\VV  \int_{a \leq a_1} \xi\cdot v \chi_0 \frac{a^{\beta +\theta+1}\Eps \xi \cdot v}{(\epsilon\xi \cdot v )^2+a^{2\beta}} 
&=\hat{\rho} \lim_{\Eps \to 0} \f{1}{\Eps^{\mu}}   \int_\VV  \int_{0}^{a_0} \xi\cdot v \frac{a^{-(\beta+n-1)}\Eps \xi \cdot v}{(\epsilon\xi \cdot va^{-\beta})^2+1}\da  \nn
\\
&= \hat\rho\f{1}{\Eps^{\mu}}    \int_\VV   c_1 |v_1| |\xi|(  |v_1| \Eps |\xi |)^\f{2-n}{\beta} 
=  \nu_0 |\xi|^{1+\mu} \hat \rho \,,
\label{calcul_frac}
\end{align}
where $v = (v_1, \cdots, v_d)$. The diffusion coefficient $\nu_0$ is given by 
\begin{align} \label{def:nu-0}
  \nu_0 = \int_\VV c_1 |v_1|^{1+\frac{2-n}{\beta}} \dv
\end{align}
with $c_1$ defined in Lemma~\ref{lem:integral}. This calculation gives the desired scale $\mu = \f{2-n}{\beta}$ and the fractional derivative in \eqref{eq:fractional}.

Next we prove that the remainder term $\hat{RJ_\Eps^1} $ defined in~\eqref{def:RJ-Eps} vanishes.  
Again we treat the two parts $a > a_1$ and $a < a_1$ separately. For $a > a_1$, using the $L^2$ bound in Lemma~\ref{lem:Fourier-bound} together with a similar estimate for deriving~\eqref{limit:J-Eps-1-1}, one can show that the part where $a > a_1$ vanishes. Therefore we may again only consider the tail $a<a_1$. This part can be controlled by using Lemma~\ref{lem:Fourier-bound}, which gives
\begin{align*}
& \f{1}{\Eps^{\mu}} \vpran{\int^{a_1}_0 \int_\VV |v|  \chi_0  \frac{a^{-\beta} \Eps |\xi \cdot v|}{(\epsilon \xi \cdot v a^{-\beta})^2+1} Q_0(a)\da \dv}
\; \vpran{\sup_a \abs{\f{ \vint{ \hat q_\Eps (t,\xi, a) }}{ Q_0(a)}-  \hat  \rho_\Eps (t,\xi)}}
 \\
&
=C\f{1}{\Eps^{\mu}} \vpran{\int^{a_1}_0 \int_\VV |v| \frac{a^{-(n-1+\beta)} \Eps |\xi \cdot v|}{(\epsilon \xi \cdot v a^{-\beta})^2+1} \da \dv}
\; \vpran{\sup_a \abs{\f{ \vint{ \hat q_\Eps (t,\xi, a) }}{ Q_0(a)}-  \hat  \rho_\Eps (t,\xi)}}
\\
&
= C \vpran{\int_\VV |v| |\xi \cdot v |^\f{2-n}{\beta}  \dv} \sup_a \abs{ \int_\VV \f{\hat q_\Eps(t,\xi, a ,v) }{ Q_0(a)} \dv-  \int_\VV \hat  R_\Eps (t,\xi,v) \dv}
 \\
& \leq C |\xi |^\f{2-n}{\beta}  \int_\VV  \sup_a \abs{ \f{\hat q_\Eps }{ Q_0(a) }-  \hat  R_\Eps } \dv
\leq C |\xi |^\f{2-n}{\beta} \int_\VV K^{1/2}(t,\xi, v) \dv
\leq C |\xi |^\f{2-n}{\beta} \vpran{\int_\VV K(t,\xi, v) \dv}^{1/2} \,.
\end{align*}
By the assumption in~\eqref{assump:main}-\eqref{value_mu} that $1+\mu > s$, the following limit holds: 
\begin{align} \label{limit:RJ-Eps-1}
   i\xi\cdot\hat {RJ_\Eps^1} \to 0 
\qquad
   \text{in} \,\, \CalD'(\R^+ \times \R^d). 
\end{align}
Combining~\eqref{limit:RJ-Eps-1} with~\eqref{calcul_frac}, we obtain that
\begin{align*}
  \hat  J^1_\Eps(t,\xi) \to \nu_0 |\xi |^{\f{2-n}{\beta} + 1}   \hat \rho 
\qquad
   \text{in $\mathcal{D}'(\R^+ \times \R^d)$} \,.
\end{align*}

Next, we show that $\hat J_\Eps^2$ in~\eqref{eq:Jdecomp} vanishes as $\Eps \to 0$. After integrating by parts,
the term $\hat J_\Eps^2$ satisfies 
\begin{align*}
\hat J_\Eps^2 & =  \Eps^{s-\mu} \int_\VV\int_0^1\frac{\vpran{i \xi \cdot v}\chi_0}{i\epsilon\xi \cdot v+\Lambda}
     \del_a\vpran{D(a)Q_0(a)\del_a \f{\hat q_\epsilon}{Q_0(a)}}\da\dv
\\
     &=   -\Eps^{s-\mu} \int_\VV\int_0^1 \left[\frac{\vpran{i \xi \cdot v} \del_a \chi_0}{i\epsilon\xi \cdot v+\Lambda} - \frac{\vpran{i \xi \cdot v}\chi_0  \del_a \Lambda }{(i\epsilon\xi \cdot v+\Lambda)^2} \right]
    D(a)Q_0(a)\del_a \f{\hat q_\epsilon}{Q_0(a)}\da\dv
\end{align*}
By the definition of $K$ in~\eqref{def:K} and the Cauchy-Schwarz inequality, we can bound $\hat J_\Eps^2$ as
\begin{align*}
 |{\hat J_\Eps^2} |^2
&\leq C \Eps^{2(s-\mu) }  \int_0^1 \int_\VV  D(a) Q_0(a)
 \left[ \frac{ |\xi \cdot v|^2 (\del_a \chi_0)^2}{| \epsilon\xi \cdot v|^2 +\Lambda^2} 
 + \frac{|\xi \cdot v|^2 \chi_0^2  (\del_a \Lambda)^2 }{((\epsilon\xi \cdot v)^2+\Lambda^2 )^2 } \right] \dv\da   \; \int_{\VV} K(t,\xi,v) \dv  
\\ 
&\Denote C \Eps^{2(s-\mu) }  \left[ G^1(t,\xi)+G^2(t,\xi) \right] \;  \int_{\VV} K(t,\xi,v) \dv .
\end{align*}
To bound the term $G^1$, we use the definitions of $\chi_0$ in \eqref{def:chi-0} and obtain
\begin{align*}
G^1(t,\xi) = \int_0^1 \int_\VV \frac{1}{D(a) Q_0(a)} \frac{ |\xi \cdot v|^2 }{ | \epsilon\xi \cdot v|^2+ \Lambda^{2}}\dv\da .
\end{align*}
Since $\frac{1}{D(a) Q_0(a)}$ is integrable on $[0, 1]$, the part $a>a_1$ contribute to a small term and we focus on the region where $a<a_1$. The corresponding contribution to $G^1$ is bounded by Lemma~\ref{lem:integral},
\begin{align*}
c\int_0^1 \int_\VV    \frac{ a^{-(2\beta+n+\theta)} |\xi \cdot v|^2 }{1+ (| \epsilon \xi \cdot v| a^{-\beta})^2}\dv\da = c \int_\VV  | \epsilon \xi \cdot v|^{\frac{1-n-\theta -2 \beta}{\beta}} |\xi \cdot v|^2  \dv.
\end{align*}
Since $\theta+n<1$, the integral in $v$ converges. Taking into account~\eqref{est:Kfourier} and~\eqref{assump:main}, the resulting power in $\Eps$ is
\begin{align*}
2(s-\mu)\;  +\frac{1-n-\theta -2 \beta}{\beta} + 1\; +\mu-s = s-\frac{\theta + 1}{\beta} -1 >0 \,.
\end{align*}
Hence the contribution of the $G^1$-term to $\hat{J_\Eps^2}$ vanishes in $L^2(\R^d)$.

The term with $G^2$ is treated similarly. 
For $a> a_1$, we use the condition for $\Lambda$ in \eqref{def:Lambda} and obtain an upper bound as
\begin{align*}
   C \int_{a> a_1} \int_\VV  D(a) Q_0(a) \ \dv\da  
< \infty
\end{align*}
Therefore the contribution to $G^2$ from the part $a > a_1$ vanishes. 

The contribution to $G^2$ for $a< a_1$ is estimated by the change of variables as follows:
\begin{align*}
 \int_{a<a_1} \int_\VV  D(a) Q_0(a) \chi_0^2  \frac{ (\del_a \Lambda)^2 |\xi \cdot v|^2 }{((\epsilon\xi \cdot v)^2+\Lambda^2 )^2 } \dv\da & \leq C
  \int_{a<a_1} \int_\VV   \frac{a ^{-(2\beta+n+\theta)}  |\xi \cdot v|^2 }{(1+(\epsilon\xi \cdot v \,a^{-\beta})^2 )^2 } \dv\da
\\
& \leq 
 C \int_\VV (\epsilon\xi \cdot v)^{\frac{1-n-\theta- 2 \beta }{\beta}} |\xi \cdot v|^2 \dv 
  = C \epsilon^{\frac{1-n-\theta- 2 \beta }{\beta}} |\xi|^{\frac{1-n - \theta}{\beta}},
 \end{align*}
which, by assumption~\eqref{assump:main}, gives the resulting total power of $\Eps$ as
\begin{align*}
 2(s-\mu)\;  +\frac{1-n-\theta - 2 \beta }{\beta} \; + 1+ \mu-s = s- \frac{\theta+1}{\beta} -1 >0
\end{align*}
Overall we have
\begin{align*}
  {\hat J_\Eps^2} \to 0 
\qquad
  \text{in} \,\, L^2(\R^+ \times \R^d) \,.
\end{align*}

Finally, we show that $\hat J_\Eps^3$ vanishes as $\Eps \to 0$. Recall the definition of $\hat J_\Eps^3$:
\begin{align*}
\hat J_\Eps^3 (t, \xi)
= -  \Eps \int_\VV\int_0^1\frac{\chi_0}{a(1-a)} \frac{\vpran{i \xi \cdot v}}{i\epsilon\xi \cdot v+\Lambda} \hat q_\Eps \da\dv \,,
\end{align*}
Similarly as before, we separate the integral as
\begin{align*}
- \hat J_\Eps^3 (t, \xi)
= \Eps \int_{\VV} \int_{a>a_1} \chi_0 \frac{\vpran{i \xi \cdot v}}{i\Eps\xi \cdot v+\Lambda} \hat q_\Eps\da\dv
+ \Eps \int_\VV \int_{a<a_1} \chi_0 \frac{\vpran{i \xi \cdot v}}{i\Eps\xi \cdot v+\Lambda} \hat q_\Eps\da\dv \,.
\end{align*}
Using the Cauchy-Schwarz inequality, we can bound the term with the integration over $a>a_1$ by
\begin{align*}
  C \Eps \int_\VV \int_0^1 |\hat q_\Eps|\da\dv 
\leq   
  \Eps \left( \int_\VV \int_0^1 \frac{ | \hat q_\Eps |^2}{Q_0} \da\dv\right)^{1/2} \,.
\end{align*}
By Lemma~\ref{lem:Fourier-bound}, this term is of order $\Eps$ in $L^2(\R^d)$ uniformly in time. 

The second term in $-\hat J_\Eps^3$ is estimated by the change of variables. 
More specifically, we apply the Cauchy-Schwarz inequality and Lemma~\ref{lem:integral} to get
\begin{align*}
 \abs{ \Eps \int_\VV \int_{a<a_1} \chi_0 \frac{\vpran{i \xi \cdot v}}{i\Eps\xi \cdot v+\Lambda} \hat q_\Eps\da\dv}^2
&\leq 
 \Eps^2  \int_\VV \int_{a<a_1}\frac{|\xi \cdot v|^2 \chi_0^2}{(\Eps\xi \cdot v)^2+\Lambda^2} Q_0 \da\dv \vpran{ \int_\VV \int_{\RR}\frac{| \hat q_\Eps|^2}{Q_0} \da\dv}
\\
&\leq
  C \Eps^2  \int_\VV \int_{a<a_1}  \frac{|\xi \cdot v|^2 |a|^{-(2 \beta+2n+\theta-2)}}{1 + (\Eps |\xi \cdot v |\, |a|^{-\beta})^2 } \da dv  \vpran{\int_\VV \int_{\RR}\frac{| \hat q_\Eps|^2}{Q_0} \da\dv}  
\\
&\leq   
  C \Eps^2  \vpran{\int_\VV (\Eps |\xi \cdot v|)^{-\frac{2\beta+2n+\theta-3}{\beta}} |\xi \cdot v|^2  \dv}
  \vpran{\int_\VV \int_{\RR}\frac{| \hat q_\Eps|^2}{Q_0} \da\dv } 
\\
&\leq   
  C (\Eps|\xi|)^{-\frac{2n+\theta-3}{\beta}} 
  \;  \int_\VV \int_{\RR}\frac{| \hat q_\Eps|^2}{Q_0} \da\dv.  
\end{align*}
Here the integrability in $a$ and $v$ are due to the assumptions in~\eqref{assump:main}, which gives 
\begin{align*}
   0<2\beta+ 2n+\theta-3<2\beta, \qquad 2n+\theta-3<0 \,.
\end{align*}
Therefore,  by the $L^2$-bound of $\hat q_\Eps$ in Lemma~\ref{lem:Fourier-bound}, we get
\begin{align*}
     {\hat J_\Eps^3} \to 0 
\qquad
  \text{in} \,\, L^2(\R^+ \times \R^d) \,.
\end{align*}
Combining the estimates for $\hat J_\Eps^1, \hat J_\Eps^2, \hat J_\Eps^3$, we conclude that~\eqref{est:RestL2} holds. 

\bibliographystyle{amsxport}
\bibliography{fractional_ref.bib}

\end{document}